\documentclass[a4paper, 11pt, reqno]{amsart}
\usepackage{amsmath,amsfonts,amssymb,amsthm,enumerate}
\usepackage{graphicx}
\usepackage{xy} \xyoption{all}
\usepackage[pagebackref,colorlinks]{hyperref}

\newtheorem{thm}{Theorem}[section]
\newtheorem{cor}[thm]{Corollary}
\newtheorem{prop}[thm]{Proposition}
\newtheorem{lem}[thm]{Lemma}

\newcommand{\gr}{\operatorname{gr}}

\theoremstyle{definition}
\newtheorem{defn}[thm]{Definition}
\newtheorem{exas}[thm]{Example}

\let\phi\varphi

\begin{document}
\title{Realizing ultragraph Leavitt path algebras as Steinberg algebras}

\maketitle
\begin{center}
R.~Hazrat\footnote{Centre for Research in Mathematics and Data Science, Western Sydney University, Australia. E-mail address: \texttt{r.hazrat@westernsydney.edu.au}}	and
T.\,G.~Nam\footnote{Institute of Mathematics, VAST, 18 Hoang Quoc Viet, Cau Giay, Hanoi, Vietnam. E-mail address: \texttt{tgnam@math.ac.vn} 
		
	\ \ {\bf Acknowledgements}:   The second author was partially supported by the Vietnam Academy of Science and
	Technology grant CT0000.02/20-21.
}
	\end{center}
	
\begin{abstract} In this article, we realize ultragraph Leavitt path algebras as Steinberg algebras. This realization allows us to use the groupoid approach to obtain structural results about these algebras. Using skew product groupoid, we show that  ultragraph Leavitt path algebras are graded von Neumann regular rings. 
We characterize strongly graded ultragraph Leavitt path algebras and show that every ultragraph Leavitt path algebra is semiprimitive.  
Moreover, we characterize irreducible representations of ultragraph Leavitt path algebras.  We also show that ultragraph Leavitt path algebras can be realized as Cuntz-Pimsner rings.	
\medskip

		
\end{abstract}

\section{Introduction}
The study of algebras associated to combinatorial objects has attracted a great deal of attention in the past years. Part of the interest in these algebras arise from the fact that many properties of the combinatorial object translate
into algebraic properties of the associated algebras and their applications to symbolic dynamics. There have been interesting examples of algebras associated to combinatorial objects among which we mention, for example, the following ones: graph $C^*$-algebras, Leavitt path algebras, higher rank graph algebras, Kumjian-Pask algebras, ultragraph $C^*$-algebras (we refer the reader to \cite{a:lpatfd} and \cite{AAS} for a more comprehensive list).

Ultragraphs were defined by Mark Tomforde in \cite{tomf:auatelaacaastg03} as an unifying approach to Exel-Laca and graph $C^*$-algebras. They have proved to be a key ingredient in the study of Morita equivalence of Exel-Laca and graph $C^*$-algebras \cite{kmst:gaelaauacutme}. Recently, Gon\c{c}alvas and Royer have established interesting connections between ultragraph $C^*$-algebras and the symbolic dynamics of shift spaces over infinite alphabets (see \cite{gr:uassoia} and \cite{gr:iaessvuatca}).

The Leavitt path algebra associated to an ultragraph was defined by Imanfar, Pourabbas and Larki in \cite{ima:tlpaou}, along with a study of graded ideal structures and a proof of a Cuntz-Krieger uniqueness type theorem. Furthermore, it was shown in \cite{ima:tlpaou} that these algebras provide examples of algebras that can not be realized as the Leavitt path algebra of a graph; that is, the class of ultragraph path algebras is strictly larger than the class of Leavitt path algebras of graphs.
This raises the question of which results about Leavitt path
algebras of graphs can be generalized to ultragraph path algebras, and whether results from the $C^*$-algebraic setting can be proved in the algebraic level. 
Recently a number of interesting results have been obtained  among which we mention, for example, the following ones. Gon\c{c}alvas and Royer \cite{gr:iaproulpa} extended to ultragraph Leavitt path algebras Chen's construction (see \cite{c:irolpa}) of irreducible representations of graph Leavitt path algebras; and in \cite{gr:saccfulpavpsgrt} they realized ultragraph  Leavitt path algebras as partial skew group rings. Using this realization they characterized artinian ultragraph Leavitt path algebras and gave simplicity criteria for these algebras. The current article
is a continuation of this direction. Building from ideas in \cite{cs:eghmesa}, where Leavitt path algebras are realized as Steinberg algebras, we realize ultragraph Leavitt path algebras as Steinberg algebras (Theorem~\ref{ULPAisSteinAl}). This is also the algebraic version of the characterization of ultragraph $C^*$-algebras as groupoid $C^*$-algebras given in \cite{mm:gaispouca}.

Steinberg algebras were introduced in \cite{stein:agatdisa} in the context of discrete inverse semigroup algebras and independently in \cite{cfst:aggolpa} as a model for Leavitt path algebras. They can be seen as discrete analogs of groupoid $C^*$-algebras, which were introduced earlier
(see, e.g., \cite{p:gisatoa, r:agatca}). This class of algebras includes group algebras, inverse semigroup
algebras and Leavitt path algebras. In recent years, there has been a lot of work around Steinberg algebras and in particular regarding their simplicity (see, e.g., \cite{bcfs:soaateg, ce:utfsa, stein:spasoegawatisa}), semiprimitivity (see, e.g, \cite{stein:spasoegawatisa}), irreducible representations (see, e.g., \cite{stein:agatdisa}), and realizing Steinberg algebras as Cuntz-Pimsner rings (see \cite{cfhl:zgracpr}). Using these results and Theorem~\ref{ULPAisSteinAl} we prove that every ultragraph Leavitt path algebra is semiprimitive (Theorem~\ref{Semiprimitivity}), and characterize irreducible representations of ultragraph Leavitt path algebras (Theorem~\ref{Irrrep4}). We provide a groupoid approach to the sufficient part of \cite[Theorem 4.8]{gr:saccfulpavpsgrt} (Theorem~\ref{Simplicity3}), which gives necessary and sufficient conditions for an ultragraph Leavitt path algebra to be simple. Moreover, using groupoid skew product, we show that  ultragraph Leavitt path algebras are graded von Neumann regular rings (Theorems~\ref{grregular1} and \ref{grregular2}).  Finally, we obtain that ultragraph Leavitt path algebras can be realized as Cuntz-Pimsner rings (Theorem~\ref{ULPAasCP-ring2}).	

\section{Preliminaries}
A \textit{groupoid} is a small category in which every morphism is invertible. It can also be viewed as a generalization of a group which has a partial binary operation. 
Let $\mathcal{G}$ be a groupoid. If $x \in \mathcal{G}$, $s(x) = x^{-1}x$ is the source of $x$ and $r(x) = xx^{-1}$ is its range. The pair $(x,y)$ is is composable if and only if $r(y) = s(x)$. The set $\mathcal{G}^{(0)}:= s(\mathcal{G}) = r(\mathcal{G})$ is called the \textit{unit space} of $\mathcal{G}$. Elements of $\mathcal{G}^{(0)}$ are units in the sense that $xs(x) = x$ and
$r(x)x = x$ for all $x\in \mathcal{G}$. For $U, V \subseteq \mathcal{G}$, we define
\begin{center}
	$UV = \{\alpha\beta \mid \alpha\in U, \beta\in V \text{\, and \,} r(\beta) = s(\alpha)\}$ and $U^{-1} = \{\alpha^{-1}\mid \alpha\in U\}$.
\end{center}

A \textit{topological groupoid} is a groupoid endowed with a topology under which the inverse map is continuous, and such that the composition is continuous with respect to the relative topology on $\mathcal{G}^{(2)} := \{(\beta, \gamma)\in \mathcal{G}^2\mid s(\beta) = r(\gamma)\}$ inherited from $\mathcal{G}^2$. An \textit{\'{e}tale groupoid} is a topological groupoid $\mathcal{G},$ whose unit space  $\mathcal{G}^{(0)}$ is locally compact Hausdorff, and such that the domain map $s$ is a local homeomorphism. In this case, the range map $r$ and the multiplication map are local homeomorphisms and  $\mathcal{G}^{(0)}$ is open in $\mathcal{G}$ \cite{r:egatq}. 

An \textit{open bisection} of $\mathcal{G}$ is an open subset $U\subseteq \mathcal{G}$ such that $s|_U$ and $r|_U$ are homeomorphisms onto an open subset of $\mathcal{G}^{(0)}$. Similar to \cite[Proposition 2.2.4]{p:gisatoa} we have that $UV$ and $U^{-1}$ are compact open bisections for all compact open bisections $U$ and $V$ of an  \'{e}tale groupoid  $\mathcal{G}$.
If in addition $\mathcal{G}$  is Hausdorff, then $U\cap V$ is a compact open bisection (see \cite[Lemma 1]{r:tgatlpa}). An \'{e}tale groupoid $\mathcal{G}$ is called \textit{ample} if  $\mathcal{G}$ has a base of compact open bisections for its topology. 

\subsection{Steinberg algebras} Steinberg algebras were introduced in \cite{stein:agatdisa} in the context of discrete inverse semigroup algebras and independently in \cite{cfst:aggolpa} as a model for Leavitt path algebras. We recall the notion of the Steinberg algebra
as a universal algebra generated by certain compact open subsets of a Hausdorff ample groupoid. Let $\mathcal{G}$ be an ample groupoid, and $K$ a field with the discrete topology. We denote by $K^{\mathcal{G}}$ the set of all functions from $\mathcal{G}$ to $K$. Canonically,  $K^{\mathcal{G}}$ has the structure of a $K$-vector space with
operations defined pointwise.

\begin{defn} 
Let $\mathcal{G}$ be an ample groupoid, and $K$ any field.   Let $A_K(\mathcal{G})$ be the $K$-vector subspace of $K^{\mathcal{G}}$ generated by the set
\begin{center}
$\{1_U\mid U$ is a compact open bisection of $\mathcal{G}\}$,
\end{center}
where $1_U: \mathcal{G}\longrightarrow K$ denotes the characteristic function on $U$. The \textit{multiplication} of $f, g\in A_K(\mathcal{G})$ is given by the convolution
\[(f\ast g)(\gamma)= \sum_{\gamma = \alpha\beta}f(\alpha)g(\beta)\] for all $\gamma\in \mathcal{G}$. The $K$-vector subspace $A_K(\mathcal{G})$, with convolution, is called the \textit{Steinberg algebra} of $\mathcal{G}$ over $K$.
\end{defn}

It is useful to note that $1_U\ast 1_V = 1_{UV}$ for compact open bisections $U$ and $V$. In particular, $1_U\ast 1_V = 1_{U\cap V}$
whenever $U$ and $V$ are compact open subsets of $\mathcal{G}^{(0)}$ (see \cite[Proposition 4.5]{stein:agatdisa}).

\subsection{Graded rings} \label{gradedringsec}
Let $\Gamma$ be a group with identity $\epsilon$. A ring $A$ (possibly without unit) is called a
\textit{$\Gamma$-graded ring} if $A = \bigoplus_{\gamma\in \Gamma}A_{\gamma}$, where each $A_{\gamma}$ is an additive subgroup of $A$ and $A_{\gamma}A_{\delta}\subseteq A_{\gamma\delta}$ for all $\gamma, \delta\in \Gamma$. By definition, $A_{\gamma}A_{\delta}$ is the additive subgroup generated by all terms $a_{\gamma}a_{\delta}$, where $a_{\gamma} \in A_{\gamma}$ and $a_{\delta} \in A_{\delta}$. The group $ A_{\gamma}$ is called the \textit{$\gamma$-homogeneous component} of $A$, and the nonzero elements of $A_{\gamma}$ are called \textit{homogeneous of degree} $\gamma$. If $a\in A$, we write
$a = \sum_{\gamma \in \Gamma}a_{\gamma}$ for the unique expression of $a$ as a sum of homogeneous terms $a_{\gamma} \in A_{\gamma}$. If $A$ is a $K$-algebra over a field $K$, then $A$ is called a $\Gamma$-graded algebra if it is a $\Gamma$-graded ring and each $A_{\gamma}$ is a $K$-subspace of $A$. A \textit{graded homomorphism} of $\Gamma$-graded rings is a homomorphism $f: A\longrightarrow B$ such that $f(A_{\gamma})\subseteq B_{\gamma}$ for all $\gamma\in \Gamma$.
If a $\Gamma$-graded ring $A = \bigoplus_{\gamma\in \Gamma}A_{\gamma}$ has the property that $A_{\gamma}A_{\delta}= A_{\gamma\delta}$ for all $\gamma, \delta\in \Gamma$, then $A$ is called \textit{strongly $\Gamma$-graded}.

For a $\Gamma$-graded ring $A$ (possibly without unit), the \emph{smash product} ring $A\#\Gamma$
is defined as the set of all formal sums $\sum_{\gamma \in \Gamma} r^{(\gamma)} p_\gamma $,
where $r^{(\gamma)}\in A$ and $p_\gamma$ are symbols. Addition is defined component-wise
and multiplication is defined by linear extension of the rule
$(rp_{\alpha})(sp_{\beta})=rs_{\alpha\beta^{-1}}p_{\beta},$  where $r,s\in A$ and $\alpha,\beta\in\Gamma$. Here $s_{\alpha\beta^{-1}}$ is the ${\alpha\beta^{-1}}$-homogeneous component of the element $s$. 

A ring  $A$ is called \emph{von Neumann regular} (or regular for short), if for any $x \in A$, we have $x \in xAx$. These rings have very rich structures and Goodearl's book~\cite{goodearlbook} is devoted to this class of rings. The graded ring $A$ is called  \emph{graded von Neumann regular} (or graded regular for short), if for any homogeneous element $x\in A$, we have $x\in xAx$. In \S\ref{skewprodgraph} we will use the fact that for a $\Gamma$-graded ring $A$, 
 $A\# \Gamma$ is graded von Neumann regular if and only if $A$ is graded von Neumann regular (see \cite[Lemma~2.3]{ahls:gsaatr}) to obtain structural results on ultragraph Leavitt path algebras.

\subsection{Graded groupoids and graded Steinberg algebras} Let $\Gamma$ be a group with identity $\epsilon$ and $\mathcal{G}$ a topological groupoid.
The groupoid $\mathcal{G}$ is called \textit{$\Gamma$-graded} if
$\mathcal{G}$ can be partitioned by clopen subsets indexed by $\Gamma$, i.e, $\mathcal{G} = \bigsqcup_{\gamma\in \Gamma}\mathcal{G}_{\gamma}$, such that $\mathcal{G}_{\gamma}\mathcal{G}_{\delta}\subseteq \mathcal{G}_{\gamma\delta}$ for all $\gamma, \delta\in \Gamma$.
The set $\mathcal{G}_{\gamma}$ is called the \textit{$\gamma$-component} of $\mathcal{G}$. We say a subset $X\subseteq \mathcal{G}$ is \textit{$\gamma$-homogeneous} if 
$X\subseteq \mathcal{G}_{\gamma}$. Equivalently, $\mathcal{G}$ is \textit{$\Gamma$-graded} if there is a continuous functor $\kappa: \mathcal{G}\longrightarrow \Gamma$, where $\Gamma$ is
regarded as a discrete group. To match the definition of $\Gamma$-grading, from the previous paragraph, one defines
$\mathcal{G}_{\gamma} = \kappa^{-1}(\gamma)$. We say that the graded groupoid $\Gamma$ is \textit{strongly $\Gamma$-graded} if $\mathcal{G}_{\gamma}\mathcal{G}_{\delta}= \mathcal{G}_{\gamma\delta}$ for all $\gamma, \delta\in \Gamma$.

Let $\mathcal{G}$ be a $\Gamma$-graded ample groupoid, and $K$ any field. Then by \cite[Lemma 3.1]{cs:eghmesa},  $A_K(\mathcal{G})$ is a $\Gamma$-graded $K$-algebra with homogeneous components
\begin{center}
$A_K(\mathcal{G})_{\gamma} = \{1_U\mid U$ is a $\gamma$-homogeneous compact open bisection of $\mathcal{G}\}$.
\end{center}

\section{Ultragraph groupoids}
In this section, based on ultragraph groupoids described in \cite{mm:gaispouca}, we realize ultragraph Leavitt path algebras as Steinberg algebras (Theorem~\ref{ULPAisSteinAl}). Consequently, we obtain criteria for an ultragraph Leavitt path algebra to be strongly graded (Corollary~\ref{strgradedULPA}), and provide a criterion for the unit space of an ultragraph groupoid to be compact (Corollary~\ref{uspacecomp}).

We begin this section by recalling some notions and notes of ultragraph theory introduced by Tomforde in \cite{tomf:auatelaacaastg03} and \cite{tomf:soua}.

An \textit{ultragraph} $\mathcal{G} = (G^0, \mathcal{G}^1, r, s)$ consists of a countable set of vertices $G^0$, a countable set of edges $\mathcal{G}^1$, and functions $s : \mathcal{G}^1 \longrightarrow G^0$ and $r : \mathcal{G}^1 \longrightarrow \mathcal{P}(G^0)\setminus \left\{ \varnothing \right\}$, where $\mathcal{P}(G^0)$ denotes the set of all subsets of $G^0$.

A vertex $v \in G^0$ is called a \textit{sink} if $s^{-1}(v)=\varnothing$ and $v$ is called an \textit{infinite emitter} if $|s^{-1}(v)|=\infty$. A \textit{singular vertex} is a vertex that is either a sink or an infinite emitter.  
A vertex $v \in G^0$ is called a \textit{regular vertex} if $0<|s^{-1}(v)|<\infty$. 

For an ultragraph $\mathcal{G} = (G^0, \mathcal{G}^1, r, s)$ we let $\mathcal{G}^0$ denote the smallest subset of $\mathcal{P}(G^0)$ that contains $\{v\}$ for all $v\in G^0$, contains $r(e)$ for all $e\in \mathcal{G}^1$, and is closed under finite unions and finite intersections. 

A \textit{finite path} in an ultragraph $\mathcal{G}$ is either an element of $\mathcal{G}^0$ or a sequence $e_1 e_2\cdots e_n$ of edges with $s(e_{i+1})\in r(e_i)$ for all $1\le i\le n-1$ and we say that the path $\alpha=e_1 e_2\cdots e_n$ has \textit{length} $|\alpha| := n$. We consider the elements of $\mathcal{G}^0$ to be paths of length $0$. We denote by $\mathcal{G}^*$ the set of all finite paths in $\mathcal{G}$. The maps $r$ and $s$ extend naturally to $\mathcal{G}^*$. Note that when $A\in \mathcal{G}^0$ we define $s(A) = r(A) = A$. An \textit{infinite path} in $\mathcal{G}$  is a sequence $e_1e_2\cdots e_n\cdots$ of edges in $\mathcal{G}$ such that $s(e_{i+1})\in r(e_i)$ for all $i\geq 1$. The set of all infinite paths in $\mathcal{G}$ is denoted by $\mathfrak{p}^{\infty}$. For $p = e_1e_2\cdots e_n\cdots \in \mathfrak{p}^{\infty}$ and $n\in \mathbb{N}$, we denote by $\tau_{\le n}(p)$ the finite path $e_1e_2\cdots e_n$, while we
denote by $\tau_{> n}(p)$ the infinite path $e_{n+1}e_{n+2}\cdots$.

If $\mathcal{G}$ is an ultragraph, then a \textit{cycle based at} $v$ in $\mathcal{G}$ is a path $\alpha=e_1e_2\cdots e_{|\alpha|}\in \mathcal{G}^*$ with $|\alpha|\ge 1$, $s(\alpha) = v$ and $v \in r(\alpha)$. An \textit{exit} for a cycle $\alpha$ is one of the following:
\begin{itemize} 
\item[(1)] an edge $e\in \mathcal{G}^1$ such that there exists an $i$ for which $s(e)\in r(e_i)$ but $e\ne e_{i+1}$.
	
\item[(2)] a sink $w$ such that $w\in r(e_i)$ for some $i$. 
\end{itemize}

In \cite{tomf:auatelaacaastg03} Mark Tomforde  introduced  the $C^*$-algebra of an ultragraph as an unifying approach to Exel-Laca and graph $C^*$-algebras. Imanfar, Pourabbas and  Larki in \cite{ima:tlpaou}, introduced the Leavitt path algebra of an ultragraph, along with a study of ideals and a proof of a
Cuntz-Krieger uniqueness type theorem.

\begin{defn}[{cf. \cite[Theorem 2.11]{tomf:auatelaacaastg03} and \cite[Definition 2.1]{ima:tlpaou}}]\label{utraLevittpathalg}
Let $\mathcal{G}$ be an ultragraph and $K$ a field. The \textit{Leavitt path algebra $L_K(\mathcal{G})$ of $\mathcal{G}$ with coefficients in $K$} is the $K$-algebra generated by the set $\{s_e, s^*_{e}\mid e\in\mathcal{G}^1\}$ $\cup\left\{p_{_A}\mid A\in\mathcal{G}^0\right\}$, satisfying the following relations for all $A,B \in \mathcal{G}^0$ and  $e, f\in\mathcal{G}^1$:
\begin{itemize} 	
\item[(1)] $p_{_\varnothing}=0, p_{_A}p_{_B} = p_{_{A\cap B}}$ and $p_{_{A\cup B}} = p_{_A}+p_{_B}-p_{_{A\cap B}}$;
\item[(2)] $p_{s(e)}s_e = s_e = s_ep_{r(e)}$ and $p_{r(e)}s_{e}^* = s_{e}^* = s_{e}^*p_{s(e)}$;
\item[(3)] $s_{e}^*s_f = \delta_{e, f}p_{r(e)}$;
\item[(4)] $p_v = \sum_{s(e)=v}s_es_{e}^*$ for any regular vertex $v$;
\end{itemize} where $p_v$ denotes $p_{_{\{v\}}}$ and $\delta$ is the Kronecker delta.
\end{defn}

We usually denote $s_A := p_{_A}$ for $A\in\mathcal{G}^0$ and $s_\alpha := s_{e_1}\cdots s_{e_n}$ for $\alpha = e_1\cdots e_n \in \mathcal{G}^*$. It is easy to see that the mappings given by $p_{_A}\longmapsto p_{_A}$ for $A\in \mathcal{G}^0$, and $s_e\longmapsto s_e$, $s^*_e\longmapsto s^*_e$ for $e\in \mathcal{G}^1$, produce an involution on the algebra $L_K(\mathcal{G})$, and for any path $\alpha = \alpha_1\cdots\alpha_n$ there exists $s^*_{\alpha} := s^*_{e_n}\cdots s^*_{e_1}$. Also, $L_K(\mathcal{G})$ has the following \textit{universal property}: if $\mathcal{A}$ is a $K$-algebra generated by a family of elements $\{b_A, c_e, c^*_e\mid A\in \mathcal{G}^0, e\in \mathcal{G}^1\}$ satisfying the relations analogous to (1) - (4)  in Definition~\ref{utraLevittpathalg}, then there always exists a $K$-algebra homomorphism $\varphi: L_K(\mathcal{G})\longrightarrow \mathcal{A}$ given by $\varphi(p_A) = b_A$, $\varphi(s_e) = c_e$ and $\varphi(s^*_e) = c^*_e$.	Moreover, in \cite[Theorem~2.9]{ima:tlpaou} the authors showed that $L_K(\mathcal{G})$  has a canonical $\mathbb{Z}$-graded structure with homogeneous components

\begin{center}	
$L_K(\mathcal{G})_n=\textnormal{Span}_K\{s_{\alpha}p_{_A}s_{\beta}^*\mid\alpha,\beta \in \mathcal{G}^*, A\in \mathcal{G}^0 \textnormal{ and } |\alpha| -|\beta|=n\}$.
\end{center}

For any $n\geq 1$, we define $$\mathfrak{p}^n := \{(\alpha, A)\mid \alpha\in \mathcal{G}^*, \, |\alpha| = n, \, A\in \mathcal{G}^0, \, A \subseteq r(\alpha)\}.$$ We specify that $(\alpha, A) = (\beta, B)$ if and only if $\alpha = \beta$ and $A= B$. We set $\mathfrak{p}^0 = \mathcal{G}^0$ and we let $\mathfrak{p} = \sqcup_{n\geq 0}\mathfrak{p}^n$. We define the length of a pair $(\alpha, A)$, denoted $|(\alpha, A)|$, to be the length of $\alpha$. We call $\mathfrak{p}$ the \textit{ultragraph space} associated with $\mathcal{G}$ and the elements of $\mathfrak{p}$ are called \textit{ultrapaths}. We may extend the range map $r$ and the source map $s$ to $\mathfrak{p}$ by the setting $r((\alpha, A)) = A$ and $s((\alpha, A)) = s(\alpha)$. Each  $A\in \mathcal{G}^0$ is regarded as an ultrapath of length zero and we define $r (A) = s(A) = A$. It
will be convenient to embed $\mathcal{G}^*$ in $\mathfrak p$ by sending $\alpha$ to $(\alpha, r(\alpha))$, if $|\alpha|\geq 1$, and by sending $A$ to $A$ for all $A\in \mathcal{G}^0$. 

We treat $\mathfrak{p}$ like a small 
category and say that a product $x \cdot y$ is defined only when $r(x)\cap s(y) \neq \varnothing$. Namely, if $x = (\alpha, A)$ and $y = (\beta, B)$, then $x \cdot y$ is defined if and only if $s(\beta)\in A$, and in this case $x \cdot y := (\alpha\beta, B)$. Also we specify that

\begin{align*}x \cdot y= \begin{cases} x \cap y &\textnormal{if }  x, y\in \mathcal{G}^0 \textnormal{ and } x\cap y \neq \varnothing, \\ y & \textnormal{if } x\in \mathcal{G}^0,\, |y|\geq 1, \textnormal{ and } x\cap s(y) \neq \varnothing \\ (\alpha, A\cap y) &\textnormal{if } y\in \mathcal{G}^0,\, |x| = |(\alpha, A)|\geq 1,  \textnormal{ and } r(x)\cap y \neq \varnothing.\end{cases}\end{align*} 

We extend the source map $s$ to $\mathfrak{p}^{\infty}$, by defining $s(\gamma) = s(e_1)$, where $\gamma = e_1e_2\cdots$ We may concatenate pairs in $\mathfrak{p}$, with infinite paths in 
$\mathfrak{p}^{\infty}$ as follows. If $y = (\alpha, A)\in \mathfrak{p}$, and if $\gamma = e_1e_2\cdots\in \mathfrak{p}^{\infty}$ are such that $s(\gamma)\in r(y) = A$, then the expression $y\cdot \gamma$ is defined to be $\alpha\gamma = \alpha e_1e_2\cdots\in \mathfrak{p}^{\infty}$. If $y = A\in \mathcal{G}^0$, we define $y\cdot \gamma = A\cdot \gamma = \gamma$ whenever $s(\gamma)\in A$. Of course $y\cdot \gamma$ is not defined if $s(\gamma) \notin r(y) = A$.

We note that $\mathcal{G}^0$ is an idempotent inverse semigroup in its own under intersetion. A \textit{filter} in $\mathcal{G}^0$ is a subsemigroup $E$ of $\mathcal{G}^0$ with the property that $\varnothing \notin E$ and if $A\in X$, and $A \subseteq B$, then $B\in X$. 
Each element $A\in \mathcal{G}^0$ determines a principal filter, denoted $\widetilde{A}$, which is the set given by: $\widetilde{A} =\{B\in \mathcal{G}^0\mid A \subseteq B\}$. 
An \textit{ultrafilter} is a filter which is not properly contained in any filter. We denote by $\mathcal{U}(\mathcal{G}^0)$  the collection of all sets in $\mathcal{G}^0$ whose principal filter in $\mathcal{G}^0$ is also an ultrafilter over $\mathcal{G}^0$. We note that $\mathcal{U}(\mathcal{G}^0)$
contains every singleton set determined by the vertices in $G^0$. We shall call the elements in $\mathcal{U}(\mathcal{G}^0)$ \textit{ultrasets}.

Following \cite[Definition~17]{mm:gaispouca}, for each subset $A$ of $G^0$, let $\epsilon(A)$ be the set $\{e\in \mathcal{G}^1\mid s(e)\in A\}$. We shall say that a set $A$ in $\mathcal{G}^0$ is an \textit{infinite emitter} if $\epsilon(A)$ is infinite.

We now describe the ultragraph groupoid associated to an ultragraph.  Let $\mathcal{G} = (G^0, \mathcal{G}^1, r, s)$ be an ultragraph without sinks. Define
\begin{center}
$Y_{\infty} := \{y\in \mathfrak{p}\mid r(y)$ is an ultraset emitting infinitely many edges$\}$.
\end{center}	Let $$X_{\mathcal{G}}: = Y_{\infty}\cup \mathfrak{p}^{\infty}$$ and 
$$\mathfrak{G}_{_\mathcal{G}} :=\{(x\cdot \mu, |x| - |y|, y\cdot \mu)\mid x, y\in \mathfrak{p}, \mu\in X_{\mathcal{G}}, r(x) = r(y), x\cdot \mu, y\cdot \mu \in X_{\mathcal{G}}\}.$$ 

We view each $(x\cdot \mu, |x| - |y|, y\cdot \mu)\in \mathfrak{G}_{_\mathcal{G}}$ as a morphism with range $x\cdot\mu$ and source $y\cdot\mu$. The formulas $(x\cdot \mu, |x| - |y|, y\cdot \mu) (y\cdot \mu, |y| - |y'|, y'\cdot \mu) = (x\cdot\mu, |x| + |y'|, y'\cdot\mu)$ and $(x\cdot \mu, |x| - |y|, y\cdot \mu)^{-1} = (y\cdot \mu, |y| - |x|, x\cdot \mu)$ define composition and inverse maps on $\mathfrak{G}_{_\mathcal{G}}$ making it a groupoid with $\mathfrak{G}_{_\mathcal{G}}^{(0)} = \{(\mu, 0, \mu)\mid \mu\in X_{\mathcal{G}}\}$ which we identify with the set $X_{\mathcal{G}}$ by the map $(\mu, 0, \mu)\longmapsto \mu$.
We note that $r_{\mathfrak{G}_{_\mathcal{G}}}$ and $s_{\mathfrak{G}_{_\mathcal{G}}}: \mathfrak{G}_{_\mathcal{G}}\longrightarrow \mathfrak{G}_{_\mathcal{G}}^{(0)}$ are the range and source maps defined respectively by: $r_{\mathfrak{G}_{_\mathcal{G}}}(x\cdot \mu, |x| - |y|, y\cdot \mu) = (x\cdot\mu, 0, x\cdot\mu)$ and $s_{\mathfrak{G}_{_\mathcal{G}}}(x\cdot \mu, |x| - |y|, y\cdot \mu) = (y\cdot\mu, 0, y\cdot\mu)$ for all $(x\cdot \mu, |x| - |y|, y\cdot \mu)\in \mathfrak{G}_{_\mathcal{G}}$.  

We next describe a topology on $\mathfrak{G}_{_\mathcal{G}}$. For $x\in \mathfrak{p}$, a finite subset $K$ of edges emitted by $r(x)$, and a finite subset $Q$ of $\mathcal{G}^0$, we define

\[\mathcal{A}(x,x) = \{(x\cdot \mu, 0, x\cdot \mu)\mid \mu\in X_{\mathcal{G}},\,  x\cdot \mu\in X_{\mathcal{G}}\}\] and 

\[\mathcal{A}(x,x, K, Q) = \mathcal{A}(x,x)\setminus (\bigcup_{e\in K}\mathcal{A}(x\cdot e,x\cdot e) \cup  \bigcup_{C\in Q}\mathcal{A}(xC,xC)).\]

For $x, y\in \mathfrak{p}$ with $r(x) = r(y)$, a finite subset $K$ of edges emitted by $r(x)$, and a finite subset $Q$ of $\mathcal{G}^0$, we define

\[\mathcal{A}(x,y) = \{(x\cdot \mu, |x| - |y|, y\cdot \mu)\mid \mu\in X_{\mathcal{G}},\,  x\cdot \mu,  y\cdot \mu\in X_{\mathcal{G}}\}\] and 

\[\mathcal{A}(x,y, K, Q) = \mathcal{A}(x,y)\setminus (\bigcup_{e\in K}\mathcal{A}(x\cdot e,y\cdot e) \cup  \bigcup_{C\in Q}\mathcal{A}(xC,yC)).\] The sets $\mathcal{A}(x,x, K, Q)$ constitute a basis of compact open sets for a locally compact Hausdorff topology on $\mathfrak{G}_{_\mathcal{G}}^{(0)}$, and the sets $\mathcal{A}(x,y, K, Q)$ constitute a basis of compact open bijections for a topology under which $\mathfrak{G}_{_\mathcal{G}}$ is a Hausdorff ample groupoid (refer to \cite[Theorem 15]{mm:gaispouca}). Moreover, $\mathfrak{G}_{_\mathcal{G}}$ comes with a canonical $\mathbb{Z}$-grading given by the functor $\kappa: \mathfrak{G}_{_\mathcal{G}}\longrightarrow \mathbb{Z}$ defined by $\kappa(x, k, x') = k$, $k\in \mathbb{Z}$.

The following result provides us with a criterion for an ultragraph groupoid to be strongly graded.
\begin{prop}\label{strgradedgroupoid}
For any ultragraph $\mathcal{G}$ without sinks, the ultragraph groupoid $\mathfrak{G}_{_\mathcal{G}}$ is strongly $\mathbb{Z}$-graded if and only if the following conditions hold:

$(1)$ $\mathcal{G}$ has no infinite emitters;

$(2)$ For every $k\in \mathbb{N}$ and every infinite path $p\in \mathfrak{p}^{\infty}$, there exists an initial subpath $x$ of $p$
and a finite path $y\in \mathcal{G}^*$ such that $r(x) = r(y)$ and $|y|- |x|= k$.
\end{prop}
\begin{proof}
($\Longrightarrow$) Firstly, assume that $\mathcal{G}$ has an infinite emitter. Then there exists a finite path $x\in X_{\mathcal{G}}$. The element $(x, 0, x)\in \mathfrak{G}_{_\mathcal{G}}^{(0)}$ cannot be factored in the form
$(x, |x|+1, \mu) (\mu, - (|x|+1), x)$, where $\mu\in X_{\mathcal{G}}$, and so $(x, 0, x)\notin (\mathfrak{G}_{_\mathcal{G}})_{|x|+1} (\mathfrak{G}_{_\mathcal{G}})_{-(|x|+1)}$. Therefore, $\mathfrak{G}_{_\mathcal{G}}$ is not strongly $\mathbb{Z}$-graded. Secondly, suppose $\mathcal{G}$ has no infinite emitters, but fails to satisfy item (2). This means there is some $k\in \mathbb{N}$, and some infinite path $p\in \mathfrak{p}^{\infty}$,
such that for every initial subpath $x$ of $p$, there does not exist a finite path $y\in\mathcal{G}^*$ having $r(y) = r(x)$ and $|y| - |x|= k$. Therefore, $(p, 0, p)\in \mathfrak{G}_{_\mathcal{G}}^{(0)}$ does not admit a factoring of the form $(p, 0, p) = (x\cdot p', - k, y\cdot p')(y\cdot p', k, x\cdot p')$. This shows that $(p, 0, p)\notin (\mathfrak{G}_{_\mathcal{G}})_{-k} (\mathfrak{G}_{_\mathcal{G}})_{k}$, so $\mathfrak{G}_{_\mathcal{G}}$ is not strongly $\mathbb{Z}$-graded.

$(\Longleftarrow)$ Suppose $\mathcal{G}$ satisfies items (1) and (2). By \cite[Proposition 25]{mm:gaispouca}, $X_{\mathcal{G}} = \mathfrak{p}^{\infty}$, and so $\mathfrak{G}_{_\mathcal{G}}^{(0)} = \{(p, 0, p)\mid p\in \mathfrak{p}^{\infty}\}$. Let $p\in \mathfrak{p}^{\infty}$ be arbitrary. For $n\geq 0$, we have $(p, n, \tau_{> n}(p)) \in (\mathfrak{G}_{_\mathcal{G}})_n$. For $n< 0$, item (2) implies that there exists an initial subpath $x$ of $p$ and a finite path $y\in \mathcal{G}^*$ such that $r(x) = r(y)$ and $|y|- |x|= -n$. We then have $(p, n, y\cdot \tau_{> |x|}(p))\in (\mathfrak{G}_{_\mathcal{G}})_n$. Therefore, $(p, 0, p)\in r_{\mathfrak{G}_{_\mathcal{G}}}((\mathfrak{G}_{_\mathcal{G}})_n)$ for all $n\in \mathbb{Z}$. By \cite[Lemma 3.1]{chr:sggasgsa}, $\mathfrak{G}_{_\mathcal{G}}$ is strongly $\mathbb{Z}$-graded, thus finishing the proof.
\end{proof}

We are now in position to provide the main result of this section.

\begin{thm}\label{ULPAisSteinAl}
Let $K$ be a field and $\mathcal{G}$ an ultragraph without sinks. Then the map $\pi_{\mathcal{G}}: L_K(\mathcal{G}) \longrightarrow A_K(\mathfrak{G}_{_\mathcal{G}})$, defined by $\pi_{\mathcal{G}}(p_A) = 1_{\mathcal{A}(A,A)}$, $\pi_{\mathcal{G}}(s_e) = 1_{\mathcal{A}((e, r(e)), r(e))}$ and 
$\pi_{\mathcal{G}}(s^*_e) = 1_{\mathcal{A}(r(e), (e, r(e)))}$, for all $A\in\mathcal{G}^0$ and $e\in \mathcal{G}^1$,  extends to a graded isomorphism.
\end{thm}
\begin{proof}
We define the elements $\{q_A\mid A\in \mathcal{G}^0\}$ and $\{t_e, t^*_e\mid e\in \mathcal{G}^1\}$ of $A_K(\mathfrak{G}_{_\mathcal{G}})$ by setting:
\begin{center}
$q_A := 1_{\mathcal{A}(A,A)}$, $t_e := 1_{\mathcal{A}((e, r(e)), r(e))}$ and $t^*_e := 1_{\mathcal{A}(r(e), (e, r(e)))}$.	
\end{center}
By repeating verbatim the corresponding argument in the proof of \cite[Proposition 20]{mm:gaispouca}, we obtain that $\{q_A, t_e, t^*_e\mid A\in \mathcal{G}^0, e\in \mathcal{G}^1\}$ satisfies the relations analogous to (1) - (4) in Definition~\ref{utraLevittpathalg}. Then, by the universal
property of $L_K(\mathcal{G})$, there exists a $K$-algebra homomorphism $\pi_{\mathcal{G}}: L_K(\mathcal{G}) \longrightarrow A_K(\mathfrak{G}_{_\mathcal{G}})$, which maps $p_A\longmapsto q_A$, $s_e\longmapsto t_e$ and $s^*_e\longmapsto t^*_e$. Since $q_A$ has degree $0$, $t_e$ has degree $1$ and $t^*_e$ has degree $-1$ for all $A\in \mathcal{G}^0$ and $e\in \mathcal{G}^1$, $\pi_{\mathcal{G}}$ is thus a $\mathbb{Z}$-graded homomorphism. Note that we always have
\begin{itemize} 
\item[(1)] $\pi_{\mathcal{G}}(s_{\alpha}p_As^*_{\beta}) = 1_{\mathcal{A}(x, y)}$, where $x = (\alpha, r(\alpha)\cap r(\beta) \cap A)$ and $y = (\beta, r(\alpha)\cap r(\beta) \cap A)$, for all $A\in\mathcal{G}^0$ and $\alpha, \beta\in \mathcal{G}^*\setminus \mathcal{G}^0$;

\item[(2)] $\pi_{\mathcal{G}}(s_{\alpha}p_A) = 1_{\mathcal{A}(x, r(x))}$, where $x = (\alpha, r(\alpha)\cap A)$, for all $A\in\mathcal{G}^0$ and $\alpha\in \mathcal{G}^*\setminus \mathcal{G}^0$.
\end{itemize}
This implies that $\pi_{\mathcal{G}}$ is surjective. We next show that $\pi_{\mathcal{G}}$ is injective. Indeed, assume that $\pi_{\mathcal{G}}$ is not injective, that means, there exists a nonzero element $x\in \ker(\pi_{\mathcal{G}})$. By the Reduction Theorem \cite[Theorem 3.2]{gon:ratrt19}, there exist elements $a, b\in L_K(\mathcal{G})$ such that either $axb = p_A \neq 0$ for some $A\in \mathcal{G}^0$, or $axb = \sum_{i=1}^nk_is_c^i \neq 0$, where $c$ is a cycle in $\mathcal{G}$ without exits.

In the first case, since $axb\in \ker{\pi_{\mathcal{G}}}$, this would imply that $q_A = \pi_{\mathcal{G}}(p_A) = 0$. On the other hand, since $\mathcal{G}$ has no sinks, $\mathcal{A}(A, A)\neq \varnothing$, so $q_A = 1_{\mathcal{A}(A,A)} \neq 0$, a contradiction.

So we are in the second case: there exists a cycle $c$ in $\mathcal{G}$ without exits such that  $axb = \sum_{i=1}^nk_is_c^i \neq 0$, where $k_i\in K$. We then have $\sum_{i=1}^nk_is_c^i \in \ker(\pi_{\mathcal{G}})$. Since $\pi_{\mathcal{G}}$ is a $\mathbb{Z}$-graded homomorphism, $\ker(\pi_{\mathcal{G}})$ is a graded ideal of $L_K(\mathcal{G})$, and so $k_is_c^i\in \ker(\pi_{\mathcal{G}})$ for all $i$. Let $j \in \{1, \hdots , n\}$ with $k_j\neq 0$. We have $p_{r(c)} = k^{-1}_j (s_c^j)^* \cdot k_j s_c^j\in \ker(\pi_{\mathcal{G}})$, so $q_{r(c)} = \pi_{\mathcal{G}}(p_{r(e)}) = 0$. Then, we produce a contradiction by repeating the argument described in the first case.

In any case, we receive a contradiction, so $\pi_{\mathcal{G}}$ is injective. Thus, $\pi_{\mathcal{G}}$ is a graded isomorphism, finishing the proof.
\end{proof}

In \cite[Theorem 2.6]{ima:tlpaou}, by constructing a representation for $L_K(\mathcal{G})$, is was shown that the elements $p_A$, $A \in \mathcal G^{0}$,  are not zero. By constructing the convolution algebra model for these algebras in Theorem~\ref{ULPAisSteinAl}, it is immediate that all monomials of the form $s_{\alpha}p_As^*_{\beta}$, where $r(\alpha)\cap r(\beta) \cap A\not = \varnothing$,  are not zero.

Realizing ultragraph Leavitt path algebras as groupoid algebras in Theorem~\ref{ULPAisSteinAl} allows us to use the results developed on the setting of Steinberg algebras to derive results on ultragraph Leavitt path algebras.

Firstly, we obtain the following criterion for an ultragraph Leavitt path algebra to be strongly graded.

\begin{cor}\label{strgradedULPA}
For any field $K$ and any ultragraph $\mathcal{G}$ without sinks, the ultragraph Leavitt path algebra $L_K(\mathcal{G})$  is strongly $\mathbb{Z}$-graded if and only if the following conditions hold:

$(1)$ $\mathcal{G}$ has no infinite emitters;

$(2)$ For every $k\in \mathbb{N}$ and every infinite path $p\in \mathfrak{p}^{\infty}$, there exists an initial subpath $x$ of $p$
and a finite path $y\in \mathcal{G}^*$ such that $r(x) = r(y)$ and $|y|- |x|= k$.	
\end{cor}
\begin{proof}
It follows from Proposition~\ref{strgradedgroupoid}, Theorem~\ref{ULPAisSteinAl} and \cite[Theorem 3.11]{chr:sggasgsa}.
\end{proof}

It is well-known that for any directed graph $E$, the unit space of the graph groupoid associated to $E$ is compact if and only if $E$ has finitely many vertices. The following fact provides us with a criterion for the unit space of an ultragraph groupoid to be compact.

\begin{cor}\label{uspacecomp}
For any ultragraph $\mathcal{G}$ without sinks, the unit space $\mathfrak{G}_{_\mathcal{G}}^{(0)}$ is compact if and only if $G^0\in \mathcal{G}^0$.
\end{cor}
\begin{proof} Let $K$ be an arbitrary field. By \cite[Theorem 4.11]{stein:agatdisa}, the unit space $\mathfrak{G}_{_\mathcal{G}}^{(0)}$ is compact if and only if the Steinberg algebra $A_K(\mathfrak{G}_{_\mathcal{G}})$ is unital. We note that $A_K(\mathfrak{G}_{_\mathcal{G}})\cong L_K(\mathcal{G})$ by Theorem~\ref{ULPAisSteinAl}. Moreover, by \cite[Lemma 2.12]{ima:tlpaou}, $L_K(\mathcal{G})$ is unital if and only if $G^0\in \mathcal{G}^0$. From these observations, we immediately obtain the statement of the theorem. 
\end{proof}

\section{Applications}
In this section, we use Theorem~\ref{ULPAisSteinAl} to investigate semiprimitivity (Theorem~\ref{Semiprimitivity}), simplicity (Theorem~\ref{Simplicity3}) and irreducible representations of ultragraph Leavitt path algebras (Theorem~\ref{Irrrep4}). Using groupoid skew product, we show that  ultragraph Leavitt path algebras are graded von Neumann regular rings (Theorems~\ref{grregular1} and \ref{grregular2}).
Also, we show that ultragraph Leavitt path algebras can be realized as Cuntz-Pimsner rings (Theorem~\ref{ULPAasCP-ring2}).

\subsection{Semiprimitivity} Recall that a ring is \textit{semiprimitive} if it has a faithful semisimple module (cf. \cite{lam:afcinr}). We investigate semiprimitivity of ultragraph Leavitt path algebras. Before doing so, we need to recall some notions. Let $\mathcal{G}$ be a groupoid and $u, v\in\mathcal{G}^{(0)}$. Define an equivalence relation on $\mathcal{G}^{(0)}$ by setting $u\sim v$ if there is an arrow
$g\in \mathcal{G}$ such that $s(g) = u$ and $r(g) = v$. An equivalence class will be called an \textit{orbit}. The group $\mathcal{G}_u = \{\gamma \in \mathcal{G}\mid u= r(\gamma) = s(\gamma)\}$ is called the \textit{isotropy group} of $\mathcal{G}$ at $u$. It is easy to verify that up to conjugation
in $\mathcal{G}$ (and hence isomorphism) the isotropy group of $u$ depends only on the orbit of $u$. The \textit{isotropy} of $\mathcal{G}$ is $\text{Iso}(\mathcal{G}): = \cup_{u\in \mathcal{G}^{(0)}}\mathcal{G}_u$.

The following lemma provides us with a complete description of isotropy groups in ultragraph groupoids.

\begin{lem}\label{Isotropygroup}
Let $\mathcal{G}$ be an ultragraph without sinks and $\mu\in X_{\mathcal{G}}$. Then the isotropy group $(\mathfrak{G}_{_\mathcal{G}})_x$ of $x:=(\mu, 0, \mu)$ is trivial unless $\mu = p \sigma\sigma\cdots$	where $p$ is a finite path in $\mathcal{G}^*$ and $\sigma$ is a cycle, in which case 
$(\mathfrak{G}_{_\mathcal{G}})_x\cong \mathbb{Z}$.
\end{lem}
\begin{proof}
An isotropy group element is of the form $g = (\mu, k, \mu)$ where $\mu = \alpha\nu = \beta\nu$	with $k = |\alpha| - |\beta|$. Moreover, $g$ is a unit unless $k\neq 0$. If $k\neq 0$, replacing $g$ by $g^{-1}$ we may assume that $|\alpha| > |\beta|$. Then $\alpha = \beta\gamma$ and $\nu = \gamma\nu = \gamma\gamma\cdots$. We thus deduce that $\mu = p\sigma\sigma\cdots$ with $p$ is a finite path and $\sigma$ is a cycle. Moreover,  if $m, n\ge 0$, then $\mu = p\sigma^m\sigma\sigma\cdots = p\sigma^n\sigma\sigma\cdots$ shows that $(\mu, |\sigma|(m-n), \mu)\in (\mathfrak{G}_{_\mathcal{G}})_x$. This implies  $(\mathfrak{G}_{_\mathcal{G}})_x= \{(\mu, |\sigma|(m-n), \mu)\mid m, n\ge 0\}$, and so $(\mathfrak{G}_{_\mathcal{G}})_x\cong \mathbb{Z}$ by the map: $(\mu, |\sigma|(m-n), \mu)\longmapsto m-n$, finishing the proof.
\end{proof}

As a corollary of Theorem~\ref{ULPAisSteinAl}, we have the following.

\begin{thm}\label{Semiprimitivity}
Let $K$ be a field and $\mathcal{G}$ an ultragraph. Then the ultragraph Leavitt path algebra $L_K(\mathcal{G})$	is semiprimitive.
\end{thm}	
\begin{proof} 
We first establish the result for ultragraphs $\mathcal{G}$ without sinks.	Let $x\in \mathfrak{G}_{_\mathcal{G}}^{(0)}$ be an arbitrary element.
By Lemma~\ref{Isotropygroup}, the isotropy group $(\mathfrak{G}_{_\mathcal{G}})_x$ is either trivial or isomorphic to $\mathbb{Z}$, and so the group algebra $K(\mathfrak{G}_{_\mathcal{G}})_x$ is either isomorphic to $K$ or isomorphic to $K[t, t^{-1}]$,  the Laurent polynomial ring in one-variable over $K$. This implies that $K(\mathfrak{G}_{_\mathcal{G}})_x$ is semiprimitive. Then, by \cite[Theorem 4.4]{stein:spasoegawatisa}, the Steinberg algebra $A_K(\mathfrak{G}_{_\mathcal{G}})$ is semiprimitive, and so $L_K(\mathcal{G})$	is semiprimitive by Theorem~\ref{ULPAisSteinAl}.

The result for arbitrary ultragraphs $\mathcal{G}$ then follows from the result of the previous paragraph, the Morita equivalence established in \cite[Theorem 10.5]{fir:meogaulpa} between $L_K(\mathcal{G})$ and $L_K(\mathcal{F})$ for an ultragraph $\mathcal{F}$ without sinks, and the preservation of semiprimitivity under Morita equivalence given in \cite[Proposition 3.2]{am:mefrwi}, thus finishing the proof.
\end{proof}	

\subsection{Irreducible representations} In \cite{gr:iaproulpa} Gon\c{c}alves and Royer extended to ultragraph Leavitt path algebras Chen's construction \cite{c:irolpa} of irreducible representations of Leavitt path algebras. We use a groupoid approach to construct irreducible representations of ultragraph Leavitt path algebras.

Let $K$ be a filed and $\mathcal{G}$ a groupoid, and $u\in\mathcal{G}^{(0)}$. Define $L_u := s^{-1}(u)$. The isotropy group $\mathcal{G}_u$ acts on the right of $L_u$. Consider the $K$-vector space $KL_u$ with basis $L_u$. The right action of $\mathcal{G}_u$ on $L_u$ induces a free right $K\mathcal{G}_u$-module structure on $KL_u$ (see \cite[Proposition 7.7]{stein:agatdisa}). Moreover, by \cite[Proposition 7.8]{stein:agatdisa}, $KL_u$ is a left $A_K(\mathcal{G})$-module with the scalar multiplication defined by: \[f\cdot x = \sum_{y\in L_u}f(yx^{-1})y\] for all $f\in A_K(\mathcal{G})$ and $x\in L_u$. It is useful to note (see \cite[Proposition~7.8]{stein:agatdisa}) that
\begin{align*}1_U\cdot x= \begin{cases} yx &\textnormal{if there is a } y\in U \textnormal{ such that } s(y) = r(x),\\  0 &\textnormal{otherwise}.\end{cases}\end{align*}	For a left $K\mathcal{G}_u$-module $V$, we define the corresponding
\textit{induced} left $A_K(\mathcal{G})$-module to be $$\text{Ind}_u(V) = KL_u\otimes_{K\mathcal{G}_u}V.$$ In \cite[Propositions 7.19 and 7.20]{stein:agatdisa} B. Steinberg obtained the following interesting facts.

\begin{thm}\label{Induction}
Let $K$ be a field, $\mathcal{G}$ an ample groupoid, and $u, v\in \mathcal{G}^{(0)}$.	Then the following statements hold:

$(1)$ $(${\cite[Proposition~7.19]{stein:agatdisa}}$)$ If $V$ is a simple left $K\mathcal{G}_u$-module, then $\rm{Ind}$$_u(V)$ is a simple left $A_K(\mathcal{G})$-module. Moreover, if $V$ and $W$ are non-isomorphic left $K\mathcal{G}_u$-modules, then $\rm{Ind}$$_u(V)\ncong \rm{Ind}$$_u(W)$.

$(2)$ $(${\cite[Proposition~7.20]{stein:agatdisa}}$)$ If $u$ and $v$ are elements in distinct orbits, then $\rm{Ind}$$_u(V)$ and $\rm{Ind}$$_v(V)$ are not isomorphic.
\end{thm}

Let $K$ be an arbitrary field and $\mathcal{G}$ an ultragraph without sinks. Let $x \in Y_{\infty}$, \textit{i.e.}, $x$ is an element in $\mathfrak{p}$ such that  $r(x)$ is an ultraset emitting infinitely many edges. We then have 
\begin{center}
$L_{(x, 0, x)} := s_{\mathfrak{G}_{_\mathcal{G}}}^{-1}((x, 0, x)) = \{(y, |y|-|x|, x)\in\mathfrak{G}_{_\mathcal{G}} \mid y \in \mathfrak{p}, \, r(y) = r(x)\}$.	
\end{center} Consider the $K$-vector space $KL_{(x, 0, x)}$ with basis $L_{(x, 0, x)}$. Then $KL_{(x, 0, x)}$ is a left $A_K(\mathfrak{G}_{_\mathcal{G}})$-module with the scalar multiplication satisfying the following:
\begin{align*}1_{\mathcal{A}(A,A)}\cdot (y, |y|-|x|, x)= \begin{cases} (y, |y|-|x|, x) &\textnormal{if } s(y) \in A,\\  0 &\textnormal{otherwise}\end{cases}\end{align*}

\begin{align*}1_{\mathcal{A}((e, r(e)), r(e))}\cdot (y, |y|-|x|, x)= \begin{cases} (ey, |y|-|x| +1, x) &\textnormal{if } s(y) \in r(e),\\  0 &\textnormal{otherwise}\end{cases}\end{align*}

\[1_{\mathcal{A}(r(e), (e, r(e)))}\cdot (r(x), 0, r(x)) = 0\]

\begin{align*}1_{\mathcal{A}(r(e), (e, r(e)))}\cdot (y, |y|-|x|, x)= \begin{cases} (y', |y|-|x|-1, x) &\textnormal{if } y = ey',\\  0 &\textnormal{otherwise}\end{cases}\end{align*}
for all $A\in \mathcal{G}^0$, $e\in \mathcal{G}^1$ and $(y, |y|-|x|, x) \in L_{(x, 0, x)}$. 

We also have that the isotropy group $(\mathfrak{G}_{_\mathcal{G}})_{(x, 0, x)}$ is trivial by Lemma~\ref{Isotropygroup}, so $K(\mathfrak{G}_{_\mathcal{G}})_{(x, 0, x)} \cong K$ and $K$ is only a simple module over $K(\mathfrak{G}_{_\mathcal{G}})_{(x, 0, x)}$. We then have $$\text{Ind}_{(x, 0, x)}(K) = KL_{(x, 0, x)}\otimes_{K(\mathfrak{G}_{_\mathcal{G}})_{(x, 0, x)}}K = KL_{(x, 0, x)}\otimes_KK \cong KL_{(x, 0, x)}$$
as $A_K(\mathfrak{G}_{_\mathcal{G}})$-modules. 
This module and the isomorphism $\pi_{\mathcal{G}}$ defined in Theorem~\ref{ULPAisSteinAl} induce  a simple left $L_K(\mathcal{G})$-module associated to the element $x\in Y_{\infty}$ as follows. Denote by $V_x$ the vector space over $K$ with basis given by all the elements $y$ in $\mathfrak{p}$ with $r(y) = r(x)$. We note that $V_x\cong KL_{(x,0,x)}$, as $K$-vector spaces, by the map: $y\longmapsto (y, |y|- |x|, x)$.
For any $A\in \mathcal{G}^0$, $e\in \mathcal{G}^1$, and $y = (e_1\cdots e_n, r(x)) \in \mathfrak{p}$, define

\begin{align*}p_A\cdot y= \begin{cases} y &\textnormal{if } s(y) \in A,\\  0 &\textnormal{otherwise}\end{cases} \quad
s_e\cdot y= \begin{cases} (ee_1\cdots e_n, r(x)) &\textnormal{if } s(y) \in r(e),\\  0 &\textnormal{otherwise}\end{cases}
\end{align*}

\begin{align*}s^*_e\cdot y= \begin{cases} (e_2\cdots e_n, r(x)) &\textnormal{if } e = e_1,\\  0 &\textnormal{otherwise,}\end{cases} \quad\quad\quad\textnormal{ and }\quad\quad\quad\quad  s^*_e\cdot r(x) = 0.\end{align*}	
Then by the isomorphism $\pi_{\mathcal{G}}$ defined in Theorem~\ref{ULPAisSteinAl}, the $K$-linear extension to all of $V_x$ of this action endows $V_x$ with the structure
of a left  $L_K(\mathcal{G})$-module.

\begin{lem}\label{Irrrep1}
Let $K$ be a field, $\mathcal{G}$ an ultragraph without sinks, and $x, y\in Y_{\infty}$. Then the following holds:

$(1)$ $V_x$ is a simple left  $L_K(\mathcal{G})$-module;

$(2)$ $V_x \cong V_y$ as left $L_K(\mathcal{G})$-modules if and only if $r(x) = r(y)$, which happens
precisely when $V_x = V_y$. Consequently, $V_x = V_{r(x)}$;

$(3)$ $\rm{End}$$_{L_K(\mathcal{G})}(V_x)\cong K$.
\end{lem}
\begin{proof}
(1) By Theorem~\ref{Induction} (1), $\text{Ind}_{(x,0,x)}(K) \cong KL_{(x,0,x)}$ is a simple left $A_K(\mathfrak{G}_{_\mathcal{G}})$-module, and so $V_x$ is a simple left $L_K(\mathcal{G})$-module by Theorem~\ref{ULPAisSteinAl}.

(2) Assume that  $V_x \cong V_y$ as left $L_K(\mathcal{G})$-modules. We then have $\text{Ind}_{(x,0,x)}(K) \cong \text{Ind}_{(y, 0, y)}(K)$. By Theorem~\ref{Induction} (2), $(x, 0,x)$ and $(y, 0, y)$ are elements in the same orbit, that means, there exists an element $g\in \mathfrak{G}_{_\mathcal{G}}$
such that $(x,0,x) = s(g)$ and $(y,0,y) = r(g)$. This implies $r(x) = r(y)$. Conversely, if $r(x) = r(y)$, then it is obvious that $V_x = V_y$.

(3) Let $f: V_x\longrightarrow V_x$ be a nonzero $L_K(\mathcal{G})$-homomorphism. Since $V_x$ is simple, $V_x = L_K(\mathcal{G}) r(x)$. Write $f(r(x)) = \sum^n_{i=1} k_i y_i\neq 0$, where $k_i\in K\setminus \{0\}$ and $y_i$'s are distinct elements in $\mathfrak{p}$ with $r(y_i) = r(x)$. By renumbering if
necessary, we may assume without loss of generality that $|y_1| \le |y_2|\le \cdots \le |y_n|$. If $|y_n| \geq 1$, then we have $0=  f(s^*_{y_n}\cdot r(x)) = s^*_{y_n}f(r(x)) = \sum^n_{i=1}k_i s^*_{y_n} y_i = k_nr(x) \neq 0$, a contradiction, and so $|y_n| =0$, that is, $y_n = r(x)$. Moreover, since $y_i$'s are distinct elements, we must have $n =1$, and $f(r(x)) = k_1r(x)$. Since $V_x = L_K(\mathcal{G}) r(x)$, $f(z) = k_1z$ for all $z\in V_x$. This shows that $\rm{End}$$_{L_K(\mathcal{G})}(V_x)\cong K$, finishing the proof.
\end{proof}

If $p$ and $q$ are infinite paths in $\mathcal{G}$, then we say that $p$ and $q$ are \textit{equivalent} (written $p \sim q$) in case there exist non-negative integers $m, n$ such that $\tau_{> m}(p) = \tau_{> n}(q)$. Clearly $\sim$ is an equivalence on $\mathfrak{p}^{\infty}$, and we let $[p]$ denote the $\sim$ equivalence class of the infinite path $p$. Let $c$ be a cycle in $\mathcal{G}$. Then the path $ccc\cdots$ is an infinite path in $\mathcal{G}$, which we denote by $c^{\infty}$. An infinite path $p$ is called a \textit{rational path} if $p\sim c^{\infty}$ for some a cycle $c$. An infinite path $p$ is called an \textit{irrational path} if $p$ is not rational. We denote by $\mathfrak{p}^{\infty}_{rat}$ and $\mathfrak{p}^{\infty}_{irr}$ the sets of rational and irrational paths, respectively.

Let $K$ be an arbitrary field and $\mathcal{G}$ an ultragraph without sinks. Let $p$ be an infinite path in $\mathcal{G}$. We then have 
\begin{center}
$L_{(p, 0, p)} := s_{\mathfrak{G}_{_\mathcal{G}}}^{-1}(p, 0, p) = \{(q, k, p)\in\mathfrak{G}_{_\mathcal{G}} \mid q\in [p],\, k \in \mathbb{Z}\}$.	
\end{center}
Consider the $K$-vector space $KL_{(p,0,p)}$ with basis $L_{(p,0,p)}$. Then $KL_{(p,0,p)}$ is a left $A_K(\mathfrak{G}_{_\mathcal{G}})$-module with the  multiplication satisfying the following:
\begin{align*}1_{\mathcal{A}(A,A)}\cdot (q, k, p)= \begin{cases} (q,k, p) &\textnormal{if } s(q) \in A,\\  0 &\textnormal{otherwise}\end{cases}\end{align*}

\begin{align*}1_{\mathcal{A}((e, r(e)), r(e))}\cdot (q,0,p)= \begin{cases} (eq, k +1, p) &\textnormal{if } s(q) \in r(e),\\  0 &\textnormal{otherwise}\end{cases}\end{align*}

\begin{align*}1_{\mathcal{A}(r(e), (e, r(e)))}\cdot (q,k,p)= \begin{cases} (\tau_{> 1}(q), k-1, p) &\textnormal{if } q = e\tau_{> 1}(q),\\  0 &\textnormal{otherwise}\end{cases}\end{align*}
for all $A\in \mathcal{G}^0$, $e\in \mathcal{G}^1$ and $(q,k,p)\in L_{(p, 0, p)}$. 

Suppose $p$ is an irrational path. We have that the isotropy group $(\mathfrak{G}_{_\mathcal{G}})_{(p,0,p)}$ is trivial by Lemma~\ref{Isotropygroup}, so $K(\mathfrak{G}_{_\mathcal{G}})_{(p,0,p)} \cong K$ and $K$ is only a simple module over $K(\mathfrak{G}_{_\mathcal{G}})_{(p,0,p)}$. We then have $$\text{Ind}_{(p,0,p)}(K) = KL_{(p,0,p)}\otimes_{K(\mathfrak{G}_{_\mathcal{G}})_{(p,0,p)}}K = KL_{(p,0,p)}\otimes_KK \cong KL_{(p,0,p)}$$
as $A_K(\mathfrak{G}_{_\mathcal{G}})$-modules. This module and the isomorphism $\pi_{\mathcal{G}}$ defined in Theorem~\ref{ULPAisSteinAl} induce a simple $L_K(\mathcal{G})$-module associated to the irrational path $p$ as follows. Denote by $V_{[p]}$ the vector space over $K$ with basis given by all infinite paths $q$ in $\mathcal{G}$ with $q\sim p$. We note that $KL_{(p,0,p)}\cong V_{[p]}$ as $K$-vector spaces by the map: $(q, k, p)\longmapsto q$. For any $A\in \mathcal{G}^0$, $e\in \mathcal{G}^1$, and $q \in [p]$, define

\begin{align*}p_A\cdot q= \begin{cases} q &\textnormal{if } s(q) \in A,\\  0 &\textnormal{otherwise}\end{cases} 
s_e\cdot q= \begin{cases} eq &\textnormal{if } s(q) \in r(e),\\  0 &\textnormal{otherwise}\end{cases}s^*_e\cdot q= \begin{cases} \tau_{> 1}(q) &\textnormal{if } e = e_1,\\  0 &\textnormal{otherwise.}\end{cases}
\end{align*}	
Then, since $KL_{(p,0,p)}$ is a left $A_K(\mathcal{G})$-module and by the isomorphism $\pi_{\mathcal{G}}$ defined in Theorem~\ref{ULPAisSteinAl}, the $K$-linear extension to all of $V_{[p]}$ of this action endows $V_{[p]}$ with the structure
of a left  $L_K(\mathcal{G})$-module.

\begin{lem}\label{Irrrep2}
Let $K$ be a field, $\mathcal{G}$ an ultragraph without sinks, and $p, q$ irrational paths in $\mathcal{G}$. Then the following holds:
	
$(1)$ $V_{[p]}$ is a simple left  $L_K(\mathcal{G})$-module;
	
$(2)$ $V_{[p]} \cong V_{[q]}$ as left $L_K(\mathcal{G})$-modules if and only if $[p] = [q]$, which happens precisely when $V_{[p]} = V_{[q]}$;
	
$(3)$ $\rm{End}$$_{L_K(\mathcal{G})}(V_{[p]})\cong K$.	
\end{lem}
\begin{proof}
(1) By Theorem~\ref{Induction} (1), $\text{Ind}_{(p,0,p)}(K) \cong KL_{(p,0,p)}$ is a simple left $A_K(\mathfrak{G}_{_\mathcal{G}})$-module, and so $V_{[p]}$ is a simple left $L_K(\mathcal{G})$-module by Theorem~\ref{ULPAisSteinAl}.
	
(2) Assume that  $V_{[p]} \cong V_{[q]}$ as left $L_K(\mathcal{G})$-modules. We then have $\text{Ind}_{(p,0,p)}(K) \cong \text{Ind}_{(q, 0, q)}(K)$. By Theorem~\ref{Induction} (2), $(p,0,p)$ and $(q, 0, q)$ are elements in the same orbit, that means, there exists an element $g\in \mathfrak{G}_{_\mathcal{G}}$
such that $(p,0,p) = s(g)$ and $(q,0,q) = r(g)$, showing that $p\sim q$, \textit{i.e.}, $[p] = [q]$. Conversely, if $[p] = [q]$, then it is obvious that $V_p = V_q$.
	
(3) Let $f: V_{[p]}\longrightarrow V_{[p]}$ be a nonzero $L_K(\mathcal{G})$-homomorphism. Since $V_{[p]}$ is simple, $V_{[p]} = L_K(\mathcal{G}) p$. Write $f(p) = \sum^n_{i=1} k_i q_i\neq 0$, where $k_i\in K\setminus \{0\}$ and $q_i$'s are distinct infinite paths in $\mathcal{G}$ with $q_i\sim p$. We claim that $n = 1$ and $q_1=p$. Otherwise, we may assume that $q_1\neq p$. Take $m$ large enough such that all the $\tau_{\le m}(q_i)$'s are pairwise distinct and that $\tau_{\le m}(q_1) \neq \tau_{\le m}(p)$. We than have $0= f(0)= f(\tau_{\le m}(q_1)\cdot p) = \tau_{\le m}(q_1)f(p) = \sum^n_{i=1}k_i \tau_{\le m}(q_1) q_i = k_1\tau_{> m}(q_1) \neq 0$, a contradiction, and so  $f(p) = k_1p$. Since $V_{[p]} = L_K(\mathcal{G}) p$, $f(z) = k_1z$ for all $z\in V_{[p]}$. This shows that $\rm{End}$$_{L_K(\mathcal{G})}(V_{[p]})\cong K$, finishing the proof.
\end{proof}

Suppose $p$ is a rational path, that means, $p\sim c^{\infty}$ for some cycle $c$ in $\mathcal{G}$. We then have $V_p = V_{c^{\infty}}$ as left $L_K(\mathcal{G})$-modules, and so we may assume that $p = c^{\infty}$. By Lemma~\ref{Isotropygroup}, the isotropy group $(\mathfrak{G}_{_\mathcal{G}})_{(c^{\infty},0,c^{\infty})}= \{(c^{\infty}, k|c|, c^{\infty})\mid k\in \mathbb{Z}\}\cong \mathbb{Z}$ by the map $(c^{\infty}, k|c|, c^{\infty})\longmapsto k$, and so $K(\mathfrak{G}_{_\mathcal{G}})_{(c^{\infty},0,c^{\infty})} \cong K[t, t^{-1}]$ by the setting: $(c^{\infty}, k|c|, c^{\infty})\longmapsto t^{k}$. We note that every simple $K[t, t^{-1}]$-module is of the form $K[t, t^{-1}]/(f(t))$, where $(f(t))$ is the ideal of $K[t, t^{-1}]$ generated by an irreducible polynomial $f(t)$.

Let $f$ be an irreducible polynomial in $K[t, t^{-1}]$. We then have $$\text{Ind}_{(c^{\infty},0,c^{\infty})}(K[t, t^{-1}]/(f))= KL_{(c^{\infty},0,c^{\infty})}\otimes_{K[t, t^{-1}]}K[t, t^{-1}]/(f)$$ is a left simple $A_K(\mathfrak{G}_{_\mathcal{G}})$-module (see Theorem~\ref{Induction}). This induces a simple left $L_K(\mathcal{G})$-module $V^f_{[c^{\infty}]}$ as follows:

$$V^f_{[c^{\infty}]} = V_{[c^{\infty}]}\otimes_{K[t, t^{-1}]}K[t, t^{-1}]/(f).$$

\begin{lem}\label{Irrrep3}
Let $K$ be a field, $\mathcal{G}$ an ultragraph without sinks, and $p, q$ irrational paths in $\mathcal{G}$. Let $f$ and $g$ be irreducible polynomials in $K[t, t^{-1}]$. Then the following holds:
	
$(1)$ $V^f_{[p]}$ is a simple left  $L_K(\mathcal{G})$-module;
	
$(2)$ $V^f_{[p]} \cong V^f_{[q]}$ as left $L_K(\mathcal{G})$-modules if and only if $[p] = [q]$, which happens precisely when $V^f_{[p]} = V^f_{[q]}$;
	
$(3)$ $V^f_{[p]} \cong V^g_{[p]}$ as left $L_K(\mathcal{G})$-modules if and only if $f = h g$ for some unit $h\in K[t, t^{-1}]$;
	
$(4)$ $\rm{End}$$_{L_K(\mathcal{G})}(V^f_{[p]})\cong K[t, t^{-1}]/(f)$.	
\end{lem}
\begin{proof}
(1) It follows from that $\text{Ind}_{(c^{\infty},0,c^{\infty})}(K[t, t^{-1}]/(f))$ is a left simple $A_K(\mathfrak{G}_{_\mathcal{G}})$-module.

(2) It is done similarly to item (2) of Lemma~\ref{Irrrep2}.

(3) We note that $K[t, t^{-1}]/(f) \cong K[t, t^{-1}]/(g)$ as $K[t, t^{-1}]$-modules if and only if $f = hg$ for some	unit $h\in K[t, t^{-1}]$. Using this note and Theorem~\ref{Induction} (1), we immediately obtain the statement.

(4) Let $\phi: V^f_{[p]}\longrightarrow V^f_{[p]}$ be a non-zero $L_K(\mathcal{G})$-homomorphism. Since $V^f_{[p]}$ is simple, $V^f_{[p]} = L_K(\mathcal{G}) (p\otimes 1)$, where $1$ is the identity of the field $K[t, t^{-1}]/(f)$. Write $\phi(p\otimes 1)= \sum^n_{i=1} q_i\otimes f_i\neq 0$, where $q_i$'s are distinct infinite paths with $q_i\sim p$ and $f_i$'s are elements in $K[t, t^{-1}]/(f)$. Similar to what as done the proof of Lemma~\ref{Irrrep2} (3), we obtain that $\phi(p\otimes 1) = p\otimes f_1 = (p\otimes 1) f_1$, and so $\phi(z) = zf_1$ for all $z\in V^f_{[p]}$, showing that $\rm{End}$$_{L_K(\mathcal{G})}(V^f_{[p]})\cong K[t, t^{-1}]/(f)$, thus finishing the proof.
\end{proof}

We should note that for each rational path $p$, we have $V^{t-1}_{[p]} =  V_{[p]}\otimes_{K[t, t^{-1}]}K[t, t^{-1}]/(t-1)\cong V_{[p]}$ as left $L_K(\mathcal{G})$-modules. Moreover, $V_{[q]}$ ($q$ is an irrational path) and $V^{t-1}_{[p]}$ ($q$ is a rational path) are non-isomorphic, since $(p,0,p)$ and $(q, 0, q)\in (\mathfrak{G}_{_\mathcal{G}})^{(0)}$ are elements in distinct orbits and by Theorem~\ref{Induction} (2). They were also constructed in \cite{gr:iaproulpa}.
Denote by $\text{Irr}(K[t, t^{-1}])$ the set of all irreducible polynomials in $K[t, t^{-1}]$. We define an equivalent relation $\equiv$ on $K[t, t^{-1}]$ as follows. For all $f, g\in K[t, t^{-1}]$, $f \equiv g$ if and only if $f = ug$ for some unit $u\in K[t, t^{-1}]$. We let $[f]$ denote the $\equiv$ equivalent class of $f$.

To summarize, we list all the simple modules over the ultragraph Leavitt path algebra, that are constructed in this subsection.

\begin{thm}\label{Irrrep4}
Let $K$ be a field and $\mathcal{G}$ an ultragraph without sinks. Then the following set \[\{V_{r(x)}\mid x \in Y_{\infty}\} \cup \{V_{[p]}\mid p\in \mathfrak{p}^{\infty}_{irr}\} \cup \{V^f_{[p]}\mid p \in \mathfrak{p}^{\infty}_{rat}, [f]\in \rm{Irr}(K[t, t^{-1}])/\equiv\}\]	consists of pairwise non-isomorphic simple left $L_K(\mathcal{G})$-modules.
\end{thm}
\begin{proof}
It follows from Theorem~\ref{Induction} and Lemmas~\ref{Irrrep1}, \ref{Irrrep2} and \ref{Irrrep3}.	
\end{proof}

We should note that the first two module types, $V_{r(x)}$ for $x\in Y_{\infty}$ and $V_{[p]}$ for $p\in \mathfrak{p}^{\infty}_{irr}$, of Theorem~\ref{Irrrep4} are also graded simple modules by \cite[Theorem 7.5]{ahls:gsaatr}.

\subsection{Leavitt path algebra of skew product ultragraphs}\label{skewprodgraph}
In \cite[\S 3]{mm:gaispouca}, in order to show that ultragraph groupoids are amenable, Marrero and Muhly realize the crossed product $C^*(\mathcal G)$ by the gauge action, $C^*(\mathcal G) \rtimes \mathbb T$, as $C^*(\mathcal G \times_1 \mathbb Z)$, which is an AF algebra. Here $\mathcal G \times_1 \mathbb Z$ is the skew product 
ultragraph. In this section, thanks to Theorem~\ref{ULPAisSteinAl},  we are able to give the algebraic version of this result, realizing $L_K(\mathcal G \times_1 \mathbb Z)$ as the smash product of ultragraph Leavitt path algebra $L_K(\mathcal G)$ with the group $\mathbb Z$. This in return allows us to show that $L_K(\mathcal G)$ is a graded regular ring. Several structural results for the algebra  $L_K(\mathcal G)$ follows (see Theorem~\ref{grregular2}).

We recall the notion of skew product of ultragraphs from \cite{mm:gaispouca}. Let $\mathcal G=(G,\mathcal G^1,r,s)$ be an ultragraph.  Denote by $\mathcal G \times_1 \mathbb Z$ the ultragraph as follows: 
\begin{gather*}
    \text{vertices of }\mathcal G \times_1 \mathbb Z  = \big\{v_n \mid v \in G^0 \text{ and } n \in \mathbb Z \big\},\\
     \text{edges of }\mathcal G \times_1 \mathbb Z = \big\{e_n \mid e\in \mathcal G^1 \text{ and } n\in \mathbb Z \big\},\\
    s(e_n) = s(e)_n,\qquad\text{ and } r(e_n) = r(e)_{n-1}.
\end{gather*}

It was shown in \cite{mm:gaispouca} that the groupoid $\mathfrak{G}_{_{\mathcal{G \times \mathbb Z}}}$ associated to the ultragraph $\mathcal G \times_1 \mathbb Z$ is isomorphic to the skew product groupoid $\mathfrak{G}_{_\mathcal{G}}\times \mathbb Z$. We combine this with \cite[Theorem 3.4]{ahls:gsaatr} which realizes the Steinberg algebra of a skew product group as a smash product algebra (see~\S\ref{gradedringsec} for the notion of smash products) to get the following result. We note that our definition of ultragraph $\mathcal G \times_1 \mathbb Z$ slightly differs from the one give in \cite[\S 3]{mm:gaispouca}. Here we decrease the indices ($r(e_n) = r(e)_{n-1}$), whereas in \cite{mm:gaispouca} the indices increase  ($r(e_n) = r(e)_{n+1}$). The reason is the way we defined the skew-product groupoid in \cite{ahls:gsaatr}, which is isomorphic to the one given in \cite{mm:gaispouca}, however was  more obviously compatible with the multiplication in the smash product.

\begin{thm}\label{grregular1}
Let $K$ be a field and $\mathcal{G}$ an ultragraph without sinks. Then there is a graded isomorphism
$L_K(\mathcal G \times_1 \mathbb Z) \cong L_K(\mathcal G) \# \mathbb Z.$
\end{thm}
\begin{proof} We have 
\[L_K(\mathcal G \times_1 \mathbb Z) \cong A_K(\mathfrak{G}_{_{\mathcal{G \times \mathbb Z}}} ) \cong A_K(\mathfrak{G}_{_\mathcal{G}}\times \mathbb Z)\cong 
A_K(\mathfrak{G}_{_\mathcal{G}})\# \mathbb Z\cong L_K(\mathcal G) \# \mathbb Z,\]
where the first and the last isomorphisms come from Theorem~\ref{ULPAisSteinAl}, the second isomorphism induces from $\mathfrak{G}_{_{\mathcal{G \times \mathbb Z}}}  \cong \mathfrak{G}_{_\mathcal{G}}\times \mathbb Z$ and the third isomorphism follows from \cite[Theorem 3.4]{ahls:gsaatr}. 
\end{proof}

For a graded ring $A$, we denote by $J^{\gr}(A)$ the graded Jacobson radical of $A$ and by $J(A)$ the usual Jacobson radical. Recall that a (graded) ring is called a (graded) semi-prime if for any (graded) ideal $I$ in $A$, $I^n\subseteq A$, $n\in \mathbb N$,  implies $I \subseteq A$.

\begin{thm}\label{grregular2}
Let $K$ be a field and $\mathcal{G}$ an ultragraph without sinks. Then $L_K(\mathcal G)$ is graded regular ring. In particular 

\begin{enumerate}[\upshape(1)]

\item Any finitely generated right (left) graded ideal of  $L_K(\mathcal G)$ is generated by one homogeneous
idempotent;

\item Any graded right (left) ideal of $L_K(\mathcal G)$ is idempotent;

\item Any graded ideal of $L_K(\mathcal G)$ is graded semi-prime;

\item   Any graded right (left) $L_K(\mathcal G)$-module is flat;

\item $J(L_K(\mathcal G))=J^{\gr}(L_K(\mathcal G)) =0$.

\end{enumerate}
\end{thm}
\begin{proof}
By \cite[Lemma~2.3]{ahls:gsaatr},  $L_K(\mathcal G)$ is graded regular if and only if  $L_K(\mathcal G) \# \mathbb Z$ is graded regular (see also~\S\ref{gradedringsec}). By Theorem~\ref{grregular1}, 
$L_K(\mathcal G) \# \mathbb Z\cong_{\gr} L_K(\mathcal G \times_1 \mathbb Z)$. The graph $\mathcal G \times_1 \mathbb Z$ is acyclic and thus $L_K(\mathcal G \times_1 \mathbb Z)$ is an ultramatricial algebra (see, e.g., \cite[Theorem 2.7]{nn:pisulpa} or \cite[Lemma 27]{mm:gaispouca}). Since the direct limit of regular rings are regular and matrix rings over fields are regular (\cite[Theorem 1.7]{goodearlbook}), ultramatricial algebras are regular and thus are graded regular. Therefore $L_K(\mathcal G) \# \mathbb Z$ is graded regular and consequently, $L_K(\mathcal G)$ is graded regular.  Now the statements (1) to (5) are the properties of a graded regular ring (see~\cite[\S 1.1.9]{hazbook}).  
\end{proof}

\subsection{Simplicity} In \cite[Theorems 4.7 and 4.8]{gr:saccfulpavpsgrt} Gon\c{c}alves and Royer gave simplicity criteria for ultragraph Leavitt path algebras via the theory of partial skew group
rings. Based on Theorem~\ref{ULPAisSteinAl} and simplicity criteria for Steinberg algebras as described in  literature (see \cite{bcfs:soaateg}, \cite{ce:utfsa} and \cite{stein:spasoegawatisa}), we provide with a groupoid approach the sufficient part of \cite[Theorem 4.8]{gr:saccfulpavpsgrt}. 
Before doing so, we need to recall some notions and useful facts.

Let $\mathcal{G}$ be a Hausdorff ample groupoid. A subset $D$ of $\mathcal{G}^{(0)}$ is called \textit{invariant} if $s(\gamma)\in D$ implies $r(\gamma)\in D$ for all $\gamma\in \mathcal{G}$. Equivalently, $D = \{r(\gamma) \mid s(\gamma)\in D\} = \{s(\gamma) \mid r(\gamma)\in D\}$. Also, $D$ is invariant if and only if its complement is invariant. We say that $\mathcal{G}$ is \textit{minimal} if $\mathcal{G}^{(0)}$ has no nontrivial open invariant subsets. We say that $\mathcal{G}$ is \textit{effective} if the interior of $\text{Iso}(\mathcal{G})\setminus \mathcal{G}^{(0)}$ is empty. We say that $\mathcal{G}$ is \textit{topologically principal} if 
$\{u\in \mathcal{G}^{(0)}\mid \mathcal{G}_u = \{u\}\}$ is dense in $\mathcal{G}^{(0)}$. We note that any Hausdorff ample groupoid being topologically principal is in fact effective,
while the converse holds if the groupoid is second-countable (see \cite[Lemma 3.1]{bcfs:soaateg}).

Let $U$ be a closed invariant subset of $\mathcal{G}^{(0)}$. We write $\mathcal{G}_U:= s^{-1}(U)$, and then $\mathcal{G}_U$ coincides
with the \textit{restriction} \[\mathcal{G}|_{U}:= \{g\in \mathcal{G}\mid s(g), \, r(g)\in U\}\] of $\mathcal{G}$ to $U$. This $\mathcal{G}_U$  is a Hausdorff ample groupoid with  the relative topology, and its unit space is $U$. Following \cite[Defintion 2.1]{cehs:iosaosegwatlpa}, a Hausdorff ample groupoid $\mathcal{G}$ is \textit{strongly effective} if for every nonempty closed invariant subset $U$ of $\mathcal{G}^{(0)}$, the groupoid $\mathcal{G}_U$ is effective. By the above note, if $\mathcal{G}_U$ is topologically principal for all nonempty closed invariant subset $U$ of $\mathcal{G}^{(0)}$, then $\mathcal{G}$ is strongly effective.

The following theorem provides us with a criterion for a Hausdorff ample groupoid to be strongly effective.

\begin{thm}[{\cite[Theorem 3.1]{cehs:iosaosegwatlpa}}]\label{Simplicity1}
	Let $K$ be a field and $\mathcal{G}$ a Hausdorff ample groupoid. Then $\mathcal{G}$ is strongly effective if and only if 
	\begin{center}
		$U\longmapsto I_U:=\{f\in A_K(\mathcal{G})\mid \rm{supp}$$(f)\subseteq \mathcal{G}_U\}$ 
	\end{center} is a lattice isomorphism from the open invariant subsets of $\mathcal{G}^{(0)}$ onto the ideals
	of $ A_K(\mathcal{G})$.	
\end{thm}

The following lemma provides us with a sufficient condition for the ultragraph groupoid of an ultragraph without sinks to be strong effective. Before doing so, we recall some useful notions introduced in \cite{kmst:uavtq}.  Let $\mathcal{G}$ be an ultragraph and $v$ a vertex.  
A \textit{first-return path based at} $v$ in $\mathcal{G}$ is a cycle $c = e_1\cdots e_n$ such that $s(c)= v$ and $s(e_i)\neq v$ for all $i\ge 2$. An ultragraph $\mathcal{G}$ satisfies  \textit{Condition} (K) if every
vertex in $\mathcal{G}$ is either the base of no first-return path or it is the base of at least two first-return paths.

\begin{lem}\label{effective}
	For any ultragraph $\mathcal{G}$ without sinks, the ultragraph groupoid $\mathfrak{G}_{_\mathcal{G}}$ is strongly effective if $\mathcal{G}$ satisfies Condition (K).
\end{lem}
\begin{proof} Assume that $\mathcal{G}$ satisfies Condition (K). Then, by \cite[Theorem 31]{mm:gaispouca} we have that 	$(\mathfrak{G}_{_\mathcal{G}})_U$ is topologically principal for all nonempty closed invariant subset $U$ of $\mathfrak{G}_{_\mathcal{G}}^{(0)}$, and so $\mathfrak{G}_{_\mathcal{G}}$ is strongly effective, thus finishing the proof.
\end{proof}

Using Theorems~\ref{ULPAisSteinAl} and~\ref{Simplicity1}, and Lemma~\ref{effective}, we immediately obtain the following.

\begin{cor}\label{Simplicity2}
	Let $K$ be a field and $\mathcal{G}$ an ultragraph without sinks. Then the ultragraph Leavitt path algebra $L_K(\mathcal{G})$	is simple if the following holds:
	
	$(1)$ The ultragraph groupoid $\mathfrak{G}_{_\mathcal{G}}$ is minimal;
	
	$(2)$ $\mathcal{G}$ satisfies Condition (K).
\end{cor}

Following \cite{tomf:soua}, if $\mathcal{G}$ is an ultragraph and $v,w\in \mathcal{G}^0$, we write $w\ge v$ to mean that there exits a path $\alpha \in \mathcal{G}^*$ with $s(\alpha)=w$ and $v\in r(\alpha)$. We say that a vertex $v$ \textit{connects to an infinite path} $\alpha = e_1\cdots e_n\cdots$ if there exists an $i\in \mathbb{N}$ such that $v\ge s(e_i)$.

We are now in position to provide the main result of this subsection, being an algebraic version of \cite[Theorem 34]{mm:gaispouca}, giving a groupoid approach to the sufficient part of \cite[Theorem 4.8]{gr:saccfulpavpsgrt}, which gives necessary and sufficient conditions for an ultragraph Leavitt path algebra to be simple.

\begin{thm}\label{Simplicity3}
	Let $K$ be a field and $\mathcal{G}$ an ultragraph without sinks. Then  the ultragraph Leavitt path algebra $L_K(\mathcal{G})$ is simple if the following holds:
	
	$(1)$ $\mathcal{G}$ satisfies Condition (K);
	
	$(2)$ Every vertex connects to every infinite path;
	
	$(3)$ If $A\in \mathcal{G}^0$ is an infinite emitter, then for every $v\in G^0$ there exists a finite path $\alpha\in \mathcal{G}^*$ such that $s(\alpha) = v$ and $A\subseteq r(\alpha)$.
\end{thm}
\begin{proof}
	Assume that $\mathcal{G}$ satisfies the three conditions (1), (2) and (3). By Corollary~\ref{Simplicity2} we just need to show that the ultragraph groupoid $\mathfrak{G}_{_\mathcal{G}}$ is minimal. So the proof given in \cite[Theorem 3.4]{mm:gaispouca} applies, thus finishing the proof.
\end{proof}

\subsection{Realizing ultragraph Leavitt path algebras as  Cuntz-Pimsner rings} In \cite[Corollary 4.6]{cfhl:zgracpr} the authors gave conditions under which the Steinberg algebra $A_K(\mathcal{G})$ associated to a $\mathbb{Z}$-graded groupoid $\mathcal{G} = \bigsqcup_{n\in \mathbb{Z}}\mathcal{G}_{n}$ can be realised as the Cuntz-Pimsner ring of an $A_K(\mathcal{G})_0$-system. Using this result and Theorem~\ref{ULPAisSteinAl} we show that ultragraph Leavitt path algebras can be realized as Cuntz-Pimsner rings. Before doing so, we need to recall some notions. A $\mathbb{Z}$-graded groupoid $\mathcal{G} = \bigsqcup_{n\in \mathbb{Z}}\mathcal{G}_{n}$ is called \textit{unperforated} if for any $n >0$ and $g\in \mathcal{G}_n$, there exist $g_1, \dots, g_n\in \mathcal{G}_1$ such that $g = g_1\cdots g_n$. 

Corollary 4.6 in \cite{cfhl:zgracpr}  states that if $\mathcal G$ is unperforated, then  there is a graded algebra isomorphism from $A_K(\mathcal{G})$ to the Cuntz--Pimsner ring of the $A_K(\mathcal{G})_0$-system $(A_K(\mathcal{G})_{-1}, A_K(\mathcal{G})_1,\psi)$, where $\psi: A_K(\mathcal{G})_{-1} \otimes A_K(\mathcal{G})_{1}\rightarrow A_K(\mathcal{G})_{0},$ induced by the usual multiplication. 

The following shows that the ultragraph groupoid associated to an ultragraph is always unperforated.

\begin{lem}\label{Unperforated}
For any ultragraph $\mathcal{G}$ without sinks, the ultragraph groupoid $\mathfrak{G}_{_\mathcal{G}}$ is unperforated.
\end{lem}
\begin{proof}
Let $n>0$ and $g \in (\mathfrak{G}_{_\mathcal{G}})_n$. We claim that there exist $g_1, \cdots, g_n\in (\mathfrak{G}_{_\mathcal{G}})_1$ such that $g = g_1\cdots g_n$. Indeed, since $g \in (\mathfrak{G}_{_\mathcal{G}})_n$, $g$ can be written in the form $g = (x\cdot \mu, n, y\cdot \mu)$, where $x, y\in \mathfrak{p}$ with $r(x) = r(y)$, $|x| - |y| = n$,  $\mu\in X_{\mathcal{G}}$, and $x\cdot \mu, y\cdot \mu\in X_{\mathcal{G}}$. Write $x = (e_1e_2\cdots e_{n+k}, A)$ and $y = (f_1\cdots f_k, B)$, where $A\subseteq r(e_{n+k})$, $B\subseteq r(f_k)$ and $A = B$. For each  $1\le i\le n-1$, let $\mu_i := (e_{i+1}\cdots e_{n+k}, A)\cdot \mu$. We have $\mu_i\in X_{\mathcal{G}}$ for all $i$. We next construct $g_i$'s as follows: for each $1\le i\le n-1$, let $$g_i = ((e_i, r(e_i)\cdot \mu_i, 1, r(e_i)\cdot \mu_i)$$ and $$g_n = ((e_n\cdots e_{n+k}, A)\cdot \mu, 1, (f_1\cdots f_k, B)\cdot \mu).$$ We then have $g_i\in (\mathfrak{G}_{_\mathcal{G}})_1$ for all $1\le i\le n$, and $g = g_1\cdots g_n$, and so the ultragraph groupoid $\mathfrak{G}_{_\mathcal{G}}$ is unperforated, thus finishing the proof. 
\end{proof}

Consequently we obtain the following result which shows that ultragraph Leavitt path algebras can be realized as Cuntz-Pimsner rings.

\begin{thm}\label{ULPAasCP-ring2}
For a field $K$ and an ultragraph $\mathcal{G}$ without sinks, the ultragraph Leavitt path algebra $L_K(\mathcal{G})$ can be realized as a Cuntz-Pimsner ring.
\end{thm}
\begin{proof}
By Lemma~\ref{Unperforated} and \cite[Corollary 4.6]{cfhl:zgracpr} the Steinberg algebra $A_K(\mathfrak{G}_{_\mathcal{G}})$ can be realized as a Cuntz-Pimsner ring. Using this observation and Theorem~\ref{ULPAisSteinAl} we immediately receive the statement, finishing the proof.	
\end{proof}

The system we gave for the realization of ultragraph Leavitt path algebra $L_K(\mathcal{G})$ comes from the groupoid presentation of these algebras. We note that one could be able to 
give another $R$-system $(J,I,\psi)$ by setting  $R:=\mathrm{span} \{p_A:  A\in \mathcal G^{0}\}$, $I:=\mathrm{span} \{s_ep_A :  e\in \mathcal G^1, A\in \mathcal G^{0} \}$, and $J:=\mathrm{span} \{p_As_e^*: e\in E^1, A\in \mathcal G^{0}\}$ similar to Leavitt path algebras (see \cite[Example 3.6]{cfhl:zgracpr}) and ultragraph $C^*$-algebras \cite[\S6]{tomf:soua}).

%
%
%
%

\vskip 0.5 cm \vskip 0.5cm {

\end{document}
\begin{thebibliography}{99}

\bibitem{a:lpatfd} G. Abrams, Leavitt path algebras: the first decade, \textit{Bull. Math. Sci.} \textbf{5} (2015), 59-120.




\bibitem{AAS} G. Abrams, P. Ara, and M. Siles Molina, \emph{Leavitt path	algebras}.   Lecture Notes in Mathematics series, Vol. 2191, Springer-Verlag Inc., 2017.

\bibitem{ahls:gsaatr}
P. Ara, R. Hazrat, H. Li, and A. Sims,  Graded Steinberg algebras and their representations, \textit{Algebra Number Theory} \textit{12} (2018), 131--172.

\bibitem{am:mefrwi} P. N. \'{A}nh, A. Louly, L. M\'{a}rki, Morita equivalence for rings without identity, \emph{Tsukuba J. Math}, \textbf{11} (1987), 1--16.

























\bibitem{bcfs:soaateg}
J. Brown, L.O. Clark, C. Farthing and A. Sims, Simplicity of algebras associated to \'{e}tale groupoids, \emph{Semigroup Forum} \textbf{88} (2014), 433--452. 

\bibitem{c:irolpa}
X.W. Chen, Irreducible representations of Leavitt path algebras, \emph{Forum Math.} \textbf{27} (2015), 549--574.

\bibitem{ce:utfsa} L. O. Clark, C. Edie-Michell, Uniqueness theorems for Steinberg algebras, \textit{Algebr. Represent. Theory} \textbf{18} (2015), 907-916.

\bibitem{cehs:iosaosegwatlpa} L. O. Clark, C. Edie-Michell, A. an Huef and A. Sims, Ideals of Steinberg algebras of strongly effective groupoids, with applications to Leavitt path algebras, \textit{Trans. Amer. Math. Soc.} \textbf{371} (2019), 5461-5486.

\bibitem{cfst:aggolpa}
L.O. Clark, C. Farthing, A. Sims, and M. Tomforde, A groupoid generalisation of Leavitt path algebras \emph{Semigroup Forum} 
\textbf{89} (2014), 501--517. 

\bibitem{cfhl:zgracpr} L. O. Clark, J. Fletcher, R. Hazrat, and H. Li, $\mathbb{Z}$-graded rings as Cuntz-Pimsner rings, \textit{J. Algebra} \textbf{536} (2019), 82--101.

\bibitem{chr:sggasgsa}
L. O. Clark, R. Hazrat, Roozbeh and S. W. Rigby, Strongly graded groupoids and strongly graded Steinberg algebras, \textit{J. Algebra} \textbf{530} (2019), 34--68.

\bibitem{cs:eghmesa}
L.O. Clark, A. Sims, Equivalent groupoids have Morita equivalent Steinberg algebras, \textit{J. Pure Appl. Algebra} \textbf{219} (2015), 2062--2075.









  
  



\bibitem{fir:meogaulpa} M. M. Firrisa, Morita equivalence of graph and ultragraph Leavitt path algebras, arXiv:2006.06521.

\bibitem{gr:uassoia} D. Gon\c{c}alves and D. Royer, Ultragraphs and shift spaces over infinite alphabets,\textit{ Bull. Sci. Math}. \textbf{141} (2017), 25-45.

\bibitem{gon:ratrt19} D. Gon\c{c}alves and D. Royer, Representation and the reduction theorem for ultragraph Leavitt path algebras, to appear in Journal of Algebraic Combinatorics (see, also, arXiv:1902.00013v1).

\bibitem{gr:saccfulpavpsgrt} D. Gon\c{c}alves and D. Royer, Simplicity and chain conditions for ultragraph Leavitt path algebras via partial skew group ring theory,  \textit{J. Aust. Math. Soc}. (2019), DOI 10.1017/S144678871900020X.

\bibitem{gr:iaessvuatca} D. Gon\c{c}alves and D. Royer, Infnite alphabet edge shift spaces via ultragraphs and their $C^*$-algebras, \textit{Int. Math. Res. Not. IMRN} 2019, no. 7, 2177--2203.

\bibitem{gr:iaproulpa} D. Gon\c{c}alves and D. Royer, Irreducible and permutative representations of ultragraph Leavitt path algebras, \textit{Forum Math}. \textbf{32} (2020), 417--431.

\bibitem{goodearlbook} K.R. Goodearl, von Neumann regular rings, 2nd ed., Krieger Publishing Co., Malabar, FL, 1991.
 
\bibitem{hazbook} R. Hazrat, Graded rings and graded Grothendieck groups, London Mathematical Society Lecture Note Series \textbf{435}, Cambridge University Press, Cambridge, 2016.
 
 
\bibitem{ima:tlpaou} M. Imanfar, A. Pourabbas and H. Larki, The Leavitt path algebras of ultragraphs, \textit{KYUNGPOOK Math. J.} \textbf{60} (2020), 21--43.

\bibitem{kmst:uavtq} T. Katsura, P. S. Muhly, A. Sims and M. Tomforde, Utragraph algebras via topological
quivers, \textit{Studia Math.} \textbf{187} (2008) 137--155.

\bibitem{kmst:gaelaauacutme} T. Katsura, P. S. Muhly, A. Sims and M. Tomforde, Graph algebras, Exel-Laca algebras, and ultragraph algebras coincide up to Morita equivalence, \textit{J. Reine Angew. Math}. \textbf{640} (2010), 135-165.







\bibitem{lam:afcinr} T. Y. Lam, \textit{A first course in noncommutative rings}, 2nd ed., Graduate Texts in Mathematics, vol. 131, Springer-Verlag, New York, 2001.


\bibitem{mm:gaispouca} 
A. E. Marrero and P. S. Muhly, Groupoid and inverse semigroup presentations of ultragraph $C^*$-algebras,  \textit{Semigroup Forum} \textbf{77} (2008), 399--422.

\bibitem{nn:pisulpa} T. G. Nam and N. D. Nam, Purely infinite simple ultragraph Leavitt path algebras, arXiv:2007.08144.














\bibitem{p:gisatoa} A. L. T. Paterson,  \textit{Groupoids, Inverse Semigroups, and their Operator Algebras}, vol. 170 of Progress in Mathematics. Springer, 1999.

\bibitem{r:agatca} J. Renault, \textit{A groupoid approach to $C^*$-algebras}, Lecture Notes in Mathematics, 793. Springer, Berlin, 1980.

\bibitem{r:egatq}  P. Resende, \'{E}tale groupoids and their quantales,  \textit{Adv. Math.} \textbf{208} (2007), 147--209.

\bibitem{r:tgatlpa}  S. W. Rigby, The groupoid approach to Leavitt path algebras,   in ``Leavitt Path Algebras and Classical $K$-Theory",  A. A. Ambily, R. Hazrat, and B. Sury eds.  Indian Statistical Institute Series, Springer (2019), pp. 23--71.

\bibitem{stein:agatdisa}
B. Steinberg, A groupoid approach to discrete inverse semigroup algebras, \emph{Adv. Math.} \textbf{223} (2010), 689--727.

\bibitem{stein:spasoegawatisa} B. Steinberg, Simplicity, primitivity and semiprimitivity of \'{e}tale groupoid algebras with applications to inverse semigroup algebras, \textit{J. Pure Appl. Algebra} \textbf{220} (2016), 1035--1054.


\bibitem{tomf:auatelaacaastg03} M. Tomforde, A unified approach to Exel-Laca algebras and C*-algebras associated to graphs, \emph{J. Operator Theory} \textbf{50} (2003), 345--368.

\bibitem{tomf:soua} M. Tomforde, Simplicity of ultragraph algebras, \emph{Indiana Univ. Math. J.} \textbf{52} (2003), 901--925.



\end{thebibliography}
